\documentclass[letterpaper, 11pt]{article}

\usepackage[margin=3.2cm]{geometry}
\usepackage[latin1]{inputenc}
\usepackage{amsmath}
\usepackage{amsfonts}
\usepackage{amssymb}
\usepackage{amsthm}
\usepackage{bm}						
\usepackage{mdwlist}					
\usepackage{graphicx}
\usepackage{fancyhdr}
\usepackage{float}
\usepackage{cite}
\usepackage{enumerate}          
\usepackage{graphics}           
\usepackage{fancybox}           
\usepackage[extension=pdf]{hyperref}

\newcommand{\ab}[1]{\ensuremath{\left[#1\right]}}
\newcommand{\abs}[1]{\ensuremath{\left|#1\right|}}

\newcommand{\cc}{\ensuremath{\mathbb{C}}}

\newcommand{\ep}{\epsilon}
\newcommand{\ee}{\mathbb{E}}

\renewcommand{\l}{\mathcal{L}}

\newcommand{\map}{\longrightarrow}
\newcommand{\nn}{\ensuremath{\mathbb{N}}}

\newcommand{\ol}[1]{\overline{#1}}

\newcommand{\pp}{\ensuremath{\mathbb{P}}}

\newcommand{\rb}[1]{\ensuremath{\left(#1\right)}}
\newcommand{\rr}{\ensuremath{\mathbb{R}}}

\newcommand{\set}[1]{\ensuremath{\left\{#1\right\}}}

\newcommand{\wt}[1]{\widetilde{#1}}
\newcommand{\zz}{\ensuremath{\mathbb{Z}}}

\newcommand{\Leb}{\mathrm{L}}

\DeclareMathOperator{\Res}{Res}

\numberwithin{equation}{section}

\theoremstyle{definition}
\newtheorem{defn}{Definition}[section]

\newtheorem{rmk}[defn]{Remark}

\theoremstyle{plain}
\newtheorem{thm}[defn]{Theorem}
\newtheorem{prop}[defn]{Proposition}
\newtheorem{lem}[defn]{Lemma}

\newcommand{\klt}{K^{LT}_{n,r}}

\newcommand{\klim}{K^{LT}_r}
\newcommand{\klimtrunc}{K^{tr}_{r,M}}
\newcommand{\kltw}{\wt K^{LT}_{n,r}}
\newcommand{\knr}{K_{n,r}}
\newcommand{\cadj}{\wt c_n}
\newcommand{\cd}{C_{\delta_1}}
\newcommand{\ld}{\ell_{\delta_2}}

\newcommand{\zcrit}{z^\ast_{c,\gamma}}
\newcommand{\zcritw}{\wt{z}^\ast_{c,\gamma}}
\newcommand{\mucrit}{\mu^\ast_{c,\gamma}}
\newcommand{\mucritw}{\wt{\mu}^\ast_{c,\gamma}}

\newcommand{\gc}{\ol g_c}

\newcommand{\vw}{\wt{v}}
\newcommand{\ww}{\wt{w}}
\newcommand{\zw}{\wt{z}}

\newcommand{\fdcont}{C^v}
\newcommand{\fdcrel}{C^v_{\text{rel}}}
\newcommand{\fdcirrel}{C^v_{\text{irrel}}}
\newcommand{\fdclim}{\ensuremath{\hat{C}^v_\infty}}
\newcommand{\fdclimtrunc}{\ensuremath{\hat{C}^v_M}}
\newcommand{\kerclim}{\ensuremath{\hat{C}^w_\infty}}
\newcommand{\kerclimtrunc}{\ensuremath{\hat{C}^w_M}}
\newcommand{\kercont}{C^w}
\newcommand{\kercontrel}{C^w_\text{rel}}
\newcommand{\kercontirrel}{C^w_\text{irrel}}

\newcommand{\h}{H_{c,\gamma}}
\newcommand{\hn}{H_{n,c,\gamma}}

\newcommand{\gamcrit}{\gamma^\ast}
\newcommand{\gcrit}[1]{\gamma^\ast_{c,#1}}

\newcommand{\bcr}{\cite{bcr}}
\newcommand{\bcf}{\cite{bcf}}

\def\R{{\mathbb R}}
\def\C{{\mathbb C}}

\def\Z{{\mathbb Z}}
\def\P{{\mathbb P}}
\def\l{{\lambda}}
\usepackage{tikz}

\title{Tracy-Widom asymptotics for a random polymer model with gamma-distributed weights}

\author{Neil O'Connell\\ University of Warwick \and Janosch Ortmann\\ University of Toronto}

\date{}

\begin{document}

\maketitle
\fancyhead[R]{\nouppercase{\leftmark}}
\fancyhead[L]{}
\bibliographystyle{acm}
\bibstyle{plain}

\begin{abstract}
\par\noindent We establish Tracy-Widom asymptotics for the partition function of a random polymer model
with gamma-distributed weights recently introduced by Sepp\"al\"ainen.  We show that the partition function of 
this random polymer can be represented within the framework of the geometric RSK correspondence and 
consequently its law can be expressed in terms of Whittaker functions.   
This leads to a representation of the law of the partition function which is amenable to asymptotic analysis. 
In this model, the partition function 
plays a role analogous to the smallest eigenvalue in the Laguerre unitary ensemble of random matrix theory. 
\end{abstract}

\section{Introduction}

Denote by $\Phi_{m,n}$ the set of `paths' of the form $\phi=\{(1,j_1),(2,j_2),\ldots,(m,j_m)\}$, where
$1\le j_1\le \cdots\le j_m\le n$, as shown in Figure \ref{fig1}.  
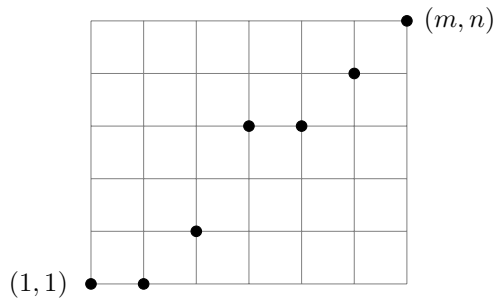
\begin{figure}
 \begin{center}
\begin{tikzpicture}[scale=.7]
\draw[help lines] (1,1) grid (7,6);

\draw [fill] (1,1) circle [radius=0.1];
\draw [fill] (2,1) circle [radius=0.1];
\draw [fill] (3,2) circle [radius=0.1];
\draw [fill] (4,4) circle [radius=0.1];
\draw [fill] (5,4) circle [radius=0.1];
\draw [fill] (6,5) circle [radius=0.1];
\draw [fill] (7,6) circle [radius=0.1];

\node at (0,1) {\small{$(1,1)$}};
\node at (8,6) {\small{$(m,n)$}};
\end{tikzpicture}

\end{center}  
\caption{ \small A path in  $\Phi_{m,n}$.  }  \label{fig1}
\end{figure}
Let $g_{ij}$ be independent gamma-distributed random variables with common parameter $\gamma$, and set
$$Z_{m,n}=\sum_{\phi\in\Phi_{m,n}} \prod_{(i,j)\in\phi} g_{ij}.$$
This is the partition function of a random polymer recently introduced by Sepp\"al\"ainen~\cite{sepp1} 
where it was observed that this model exhibits the so-called Burke property.  The analogous property for other polymer
models, specifically the semi-discrete Brownian polymer introduced in \cite{OConnellYor01} and the log-gamma polymer
introduced in \cite{Sep12}, has been used to study asymptotics of the partition function~\cite{OConnellYor01,MoriartyOConnell07,Sep12,SepValko10}.
More recently, the semi-discrete and log-gamma polymer models have been shown to have
an underlying integrable structure, via a remarkable connection between a combinatorial structure 
known as the geometric RSK correspondence and $GL(n,\mathbb R)$-Whittaker functions 
\cite{OConnellTodaLattice,cosz,osz}.  This integrable structure has allowed very precise (Tracy-Widom) asymptotics to be 
obtained \cite{bc,bcf,bcr}. For these models, the partition functions play a role analogous to the
largest eigenvalue in the Gaussian and Laguerre unitary ensembles of random matrix theory.

In the present paper, we show that the partition function of the above random polymer can also be represented
within the framework of the geometric RSK correspondence and consequently its law can be expressed 
in terms of Whittaker functions.  For this model, the partition function plays a role analogous to the {\em smallest}
eigenvalue in the Laguerre unitary ensemble.  
This leads to a representation of the law of the partition function from which we establish 
Tracy-Widom asymptotics for this model.  A precise statement is given as follows.
\begin{thm} \label{main} 
Suppose $m/n\to\alpha>0$ as $n\to\infty$.  Set $c=1+\alpha$,
$$\mu=\inf_{z>0}\big[c\psi'(z+\gamma)-\psi'(z)\big],
\qquad H(z)=\ln\Gamma(z)-c\ln\Gamma(z+\gamma)+\mu z.$$
The infimum is achieved at some $z^*>0$ and $\bar g :=-H'''(z^*)>0$.
For $\gamma$ sufficiently small,
$$\lim_{n\to\infty} \pp\set{\frac{\ln Z_{m,n}-n\mu}{n^{1/3}}\leq r} = F_\text{GUE}\rb{\rb{\ol g/2}^3\, r}$$
where $F_\text{GUE}$ is the Tracy--Widom distribution function.
\end{thm}

The connection to random matrices can be further illustrated by considering the zero-temperature limit,
which corresponds to letting $\gamma\to 0$.  Then the collection of random variables $-\gamma \log g_{ij}$ 
converge weakly to a collection of independent standard exponentially distributed variables 
$w_{ij}$ and so, by the principle of the largest term, 
the sequence $-\gamma\log Z_{m,n}$ converges weakly to the first passage percolation variable
$$f_{m,n}=\min_{\phi\in\Phi_{m,n}} \sum_{(i,j)\in\phi} w_{ij}.$$
This first passage percolation problem was previously considered in~\cite{OConnell99} where it is argued, 
using a representation of $f_{m,n}$ as a departure process from a series of `Exp/Exp/1' queues in tandem
together with the Burke property for such queues, that, almost surely,
\begin{equation}\label{fpp}
\lim_{n\to\infty} f_{\alpha n,n}/n = \left(\sqrt{1+\alpha}-1\right)^2.
\end{equation}
Moreover, it can be inferred from further results presented in~\cite{DMO05} on a discrete version of this model 
with geometric weights (or alternatively from Section 2 below) that $f_{m,n}$ has the same law as the smallest 
eigenvalue in the Laguerre ensemble with density proportional to
$$\prod_{1\le i<j\le n}(\l_i-\l_j)^2 \prod_{i=1}^n \l_i^{m-1} e^{-\l_i} d\l_i .$$
Given this identity in law, the asymptotic relation \eqref{fpp} can also be seen as a consequence of the 
Marchenko-Pastur law.  As a further consistency check, one can easily verify (see Lemma~\ref{lem:ExpandMu} below) that 
$$-\gamma\mu\to  \left(\sqrt{1+\alpha}-1\right)^2$$
as $\gamma\to 0$, where $\mu$ is defined in the statement of Theorem~\ref{main}.

The outline of the paper is as follows.  In the next section we relate the above polymer model to the 
geometric RSK correspondence and deduce, using results from \cite{cosz,osz}, an integral formula
for the Laplace transform of the partition function. In Section \ref{sec:FD} we show that this Laplace transform can be written as a Fredholm determinant, which allows us, in Section \ref{sec:asymptotics} 
to take the limit as $n\to\infty$. Section \ref{sec:proofs} contains proofs of some lemmas that we require on the way.

\bigskip

\noindent {\bf Acknowledgements.}  Thanks to Timo Sepp\"al\"ainen for helpful discussions and
for making the manuscript~\cite{sepp1} available to us.

\section{Geometric RSK, polymers and Whittaker functions}

The geometric RSK correspondence is a bijective mapping $$T:(\R_{>0})^{h\times n}\to (\R_{>0})^{h\times n}.$$
It was introduced by Kirillov~\cite{Kirillov01} as a geometric lifting of the RSK correspondence,  
and defined as follows.  
Let $W=(w_{ij})\in(\R_{>0})^{h\times n}$
and write $T(W)=(t_{ij})\in(\R_{>0})^{h\times n}$. For $1\le k\le n$ and $1\le r\le h\wedge k$,
\begin{equation}\label{grsk} 
t_{h-r+1,k-r+1}\dotsm t_{h-1,k-1} t_{hk} 
= \sum_{(\pi_1,\ldots,\pi_r)\in\Pi^{(r)}_{h,k}} \prod_{(i,j)\in \pi_1\cup\cdots\cup\pi_r} w_{ij},
\end{equation}
where $\Pi^{(r)}_{h,k}$ denotes the set of $r$-tuples of non-intersecting up/right lattice paths $\pi_1,\ldots,\pi_r$
starting at positions $(1,1),(1,2),\ldots,(1,r)$ and ending at positions $(h,k-r+1),\ldots,(h,k-1),(h,k)$, as shown 
in Figure \ref{fig2}.  The remaining entries of $T(W)$ are determined by the relation $T(W^t)=T(W)^t$. 
\setlength{\unitlength}{1.3pt}

\begin{figure}
 \begin{center}
\begin{tikzpicture}[scale=.7]
\draw[help lines] (1,1) grid (9,8);
\draw[very thick] (1,1)--(2,1)--(3,1)--(3,2)--(4,2)--(4,3)--(5,3)--(6,3)--(7,3)--(7,4)--(8,4)--(9,4);
\draw[very thick] (1,2)--(2,2)--(2,3)--(2,4)--(3,4)--(4,4)--(5,4)--(5,5)--(6,5)--(7,5)--(8,5)--(9,5);
\draw[very thick] (1,3)--(1,4)--(1,5)--(2,5)--(3,5)--(3,6)--(4,6)--(5,6)--(6,6)--(7,6)--(8,6)--(9,6);
\draw [fill] (1,1) circle [radius=0.1];
\draw [fill] (1,2) circle [radius=0.1];
\draw [fill] (1,3) circle [radius=0.1];
\draw [fill] (9,4) circle [radius=0.1];
\draw [fill] (9,5) circle [radius=0.1];
\draw [fill] (9,6) circle [radius=0.1];

\node at (0,1) {\small{$(1,1)$}};
\node at (10,6) {\small{$(h,k)$}};
\node at (10,8) {\small{$(h,n)$}};
\end{tikzpicture}

\end{center}  
\caption{ \small A 3-tuple of non-intersecting paths in $\Pi^{(3)}_{h,k}$.  }  \label{fig2}
\end{figure}
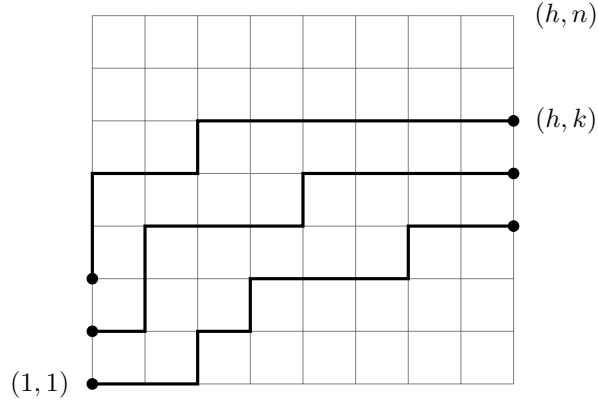

Note in particular that
$$t_{hn}=\sum_{\pi\in\Pi_{h,n}} \prod_{(i,j)\in\pi} w_{ij},$$
where $\Pi_{h,n}$ is the set of up/right lattice paths in $\Z^2$ from $(1,1)$ to $(h,n)$.
This gives an interpretation of $t_{hn}$ as a polymer partition function, providing the basis
for the analysis of the log-gamma polymer developed in~\cite{cosz,osz}.

The relation to the random polymer defined in the introduction is as follows.  
\begin{prop}
Suppose $h\ge n$ and set $m=h-n+1$.
For $1\le i\le m$ and $1\le j\le n$, set $g_{ij}=1/w_{i+j-1,n-j+1}$.  Then 
\begin{equation}\label{random}
\frac1{t_{m1}}=\sum_{\phi\in\Phi_{m,n}} \prod_{(i,j)\in\phi} g_{ij},
\end{equation}
where $\Phi_{m,n}$ the set of $\phi=\{(1,j_1),(2,j_2),\ldots,(m,j_m)\}$ with
$1\le j_1\le \cdots\le j_m\le n$.
\end{prop}
\begin{proof}  From the definition \eqref{grsk}, taking $k=r=n$,
$$ t_{m1}\dotsm t_{h-1,n-1} t_{hn} 
= \sum_{(\pi_1,\ldots,\pi_n)\in\Pi^{(n)}_{h,n}} \prod_{(i,j)\in \pi_1\cup\cdots\cup\pi_n} w_{ij}
= \prod_{i,j} w_{ij},
$$
and, taking $k=n$ and $r=n-1$,
$$ t_{m+1,2}\dotsm t_{h-1,n-1} t_{hn} 
= \sum_{(\pi_1,\ldots,\pi_{n-1})\in\Pi^{(n-1)}_{h,n}} \prod_{(i,j)\in \pi_1\cup\cdots\cup\pi_{n-1}} w_{ij}.
$$
Thus,
$$\frac1{t_{m1}} =\sum_{(\pi_1,\ldots,\pi_{n-1})\in\Pi^{(n-1)}_{h,n}} \prod_{(i,j)\notin \pi_1\cup\cdots\cup\pi_{n-1}} \frac1{w_{ij}}
= \sum_{\phi\in\Phi_{m,n}} \prod_{(i,j)\in\phi} g_{ij},$$
as required. The last identity is illustrated in Figures \ref{fig1} and \ref{fig3}.
\end{proof}
\begin{rmk} The identity \eqref{random} is analogous to Theorem 5.1, equation (5.4), of the paper \cite{DMO05},
where the corresponding identity for the usual RSK correspondence is given.
\end{rmk}
\begin{figure}
 \begin{center}
\begin{tikzpicture}[scale=.7]
\draw[help lines] (1,1) grid (12,6);
\draw[very thick] (1,5)--(2,5)--(3,5)--(3,6)--(4,6)--(5,6)--(6,6)--(7,6)--(8,6)--(9,6)--(10,6)--(11,6)--(12,6);
\draw[very thick] (1,4)--(2,4)--(3,4)--(4,4)--(5,4)--(5,5)--(6,5)--(7,5)--(8,5)--(9,5)--(10,5)--(11,5)--(12,5);
\draw[very thick] (1,3)--(2,3)--(3,3)--(4,3)--(5,3)--(6,3)--(6,4)--(7,4)--(8,4)--(9,4)--(10,4)--(11,4)--(12,4);
\draw[very thick] (1,2)--(2,2)--(3,2)--(4,2)--(5,2)--(6,2)--(7,2)--(8,2)--(9,2)--(9,3)--(10,3)--(11,3)--(12,3);
\draw[very thick] (1,1)--(2,1)--(3,1)--(4,1)--(5,1)--(6,1)--(7,1)--(8,1)--(9,1)--(10,1)--(11,1)--(11,2)--(12,2);

\draw [fill] (1,1) circle [radius=0.1];
\draw [fill] (1,2) circle [radius=0.1];
\draw [fill] (1,3) circle [radius=0.1];
\draw [fill] (1,4) circle [radius=0.1];
\draw [fill] (1,5) circle [radius=0.1];

\draw [fill] (12,2) circle [radius=0.1];
\draw [fill] (12,3) circle [radius=0.1];
\draw [fill] (12,4) circle [radius=0.1];
\draw [fill] (12,5) circle [radius=0.1];
\draw [fill] (12,6) circle [radius=0.1];

\draw (12,1) circle [radius=0.1];
\draw (10,2) circle [radius=0.1];
\draw (8,3) circle [radius=0.1];
\draw (7,3) circle [radius=0.1];
\draw (4,5) circle [radius=0.1];
\draw (1,6) circle [radius=0.1];
\draw (2,6) circle [radius=0.1];

\node at (0,1) {\small{$(1,1)$}};
\node at (13,6) {\small{$(h,n)$}};
\end{tikzpicture}

\end{center}  
\caption{ \small An $(n-1)$-tuple of non-intersecting paths in  $\Pi^{(n-1)}_{h,n}$, and its compliment.  
The corresponding path in $\Phi_{h-n+1,n}$ is shown in Figure \ref{fig1}.}  \label{fig3}
\end{figure}

Let $a\in\rr^n$ and $b\in\rr^h$ be such that $a_j+b_i>0$ for all $i,j$.
In~\cite{cosz} (here we are using the notation of \cite{osz}) it was shown that, if the matrix $W$ is chosen at 
random according to the probability measure
$$\P(dW)=\prod_{i,j}\Gamma\rb{a_j+b_i}^{-1} e^{-1/w_{ij}} w_{ij}^{a_j+b_i-1} dw_{ij}$$ 
then the law of the vector $(t_{h1},\ldots,t_{m1})$ under $\P$ is given by
$$\mu_{n}(d x)  = \prod_{i,j}\Gamma\rb{a_j+b_i}^{-1} \Psi_a^n(x)\Psi_{b;1}^n(x)\prod_{j=1}^n \frac{dx_j}{x_j},$$
where $\Psi^n_a$ and $\Psi^n_{b;1}$ are (generalised) Whittaker functions, as defined in~\cite{osz}.  
Without loss of generality
we can assume that $a_j>0$ and $b_i>0$ for each $i,j$ and deduce the following.
\begin{prop}  For $s\in\C$ with $\Re s>0$,
\begin{equation}
\label{eq:LT}
	\ee e^{-s/t_{m1}}=\int_{(\R_{>0})^n} e^{-s/x_n}\, \mu_{n}(dx) = 
\int_{(i\rr)^n}\prod_{i,j=1}^n\Gamma\rb{a_i-\lambda_j} \prod_{j=1}^n 
\frac{s^{\lambda_j}\,\prod_{i=1}^h\Gamma\rb{b_i+\lambda_j}}{s^{a_j}\,\prod_{i=1}^h\Gamma\rb{b_i+a_j}}\, 
s_n(\lambda)d\lambda
\end{equation}
where $s_n$ is the density of the Sklyanin measure
\begin{equation}
\label{eq:DefSklyanin}
	s_n(\lambda) =  \frac1{\rb{2\pi i}^n n!}\, \prod_{i,j=1}^n \frac1{\Gamma\rb{\lambda_i-\lambda_j}}.
\end{equation}
\end{prop}
\begin{proof} By \cite[Corollary 3.8]{osz} the functions $\Psi^n_{a;s}(x)\equiv e^{-s/x_n}\Psi^n_a(x)$ and 
$\Psi^n_{b;1}$ are both in $L_2((\rr_{>0})^n,\prod_{j=1}^n dx_j/x_j)$ and, by \cite[Corollary 3.5]{osz}, for 
$\l\in(i\rr)^n$, we have
$$\int_{(\R_{>0})^n} \Psi^n_{b;1}(x)  \Psi_{\l}^n(x) \prod_{j=1}^n \frac{dx_j}{x_j} = 
\prod_{i=1}^h\prod_{j=1}^n\Gamma\rb{b_i+\lambda_j}$$
and
$$\int_{(\R_{>0})^n} \Psi^n_{a;s}(x)  \Psi_{-\l}^n(x) \prod_{j=1}^n \frac{dx_j}{x_j} = s^{\sum_{j=1}^n (\l_j-a_j)}
\prod_{i,j=1}^n\Gamma\rb{a_i-\lambda_j}.$$
The claim now follows from the Plancherel theorem for $GL(n)$-Whittaker functions due to Wallach,
noting that $\overline{\Psi_{\l}^n(x) }= \Psi_{-\l}^n(x)$ 
(see for example~\cite[Section 2]{osz}).
\end{proof}

The Laplace transform of the partition function $Z_{m,n}$ of the random polymer (defined in the introduction) 
is obtained by setting $a_i=\epsilon$ and $b_j=\gamma-\epsilon$, where $0<\epsilon<\gamma$, for in this
case $Z_{m,n}$ has the same law as $1/t_{m1}$.

\section{Fredholm determinant representation}
\label{sec:FD}
The first step in the proof of Theorem \ref{main} is to write the right-hand side of \eqref{eq:LT} as a Fredholm determinant. 
A similar algebraic identity is proven in \bcr, but doesn't apply to the present setting.  We present here a self contained proof
which is an adaptation of the proof given in \bcr..  

For $s\in\rr$ we define a function $F_s$ by
\begin{align}
	\label{eq:DefF}
	F_s(w) & = s^w \prod_{j=1}^h \Gamma\rb{b_j+w}
\end{align}
where $s$ is a parameter to be chosen later (for the LLN). For $\delta>0$ define $\ell_\delta=\delta+i\rr$ and let $C_\delta$ be the circle centred at the origin of radius $\delta$.

\begin{prop}
	\label{prop:LTtoFD}
	Let $\delta_1,\delta_2>0$ such that $\delta_1<\delta_2\wedge\rb{1-\delta_2}$. Suppose also that $b_j>\delta_2$ for all $j$. Then
	\begin{align}
		\label{eq:LTtoFD}
		&\int e^{-s/x_n}\, \mu_n(dx)  = \det\rb{I+\klt}_{L^2\rb{\cd}}
		\intertext{where}
		\label{eq:DefKLT}
		\klt\rb{v,\wt v} &= \frac1{2\pi i} \int_{\ld} \frac{dw}{w-\wt v} \frac\pi{\sin\rb{\pi\rb{v-\wt v}}}\frac{F_s(w)}{F_s(v)}\,\prod_{j=1}^n \frac{\Gamma\rb{v-a_j}}{\Gamma\rb{w-a_j}}.
	\end{align}
\end{prop}

\par\noindent The rest of this section is devoted to the proof of this proposition. We will begin with the right-hand side of \eqref{eq:LTtoFD} and show that it equals to the right-hand side of \eqref{eq:LT}. 

\paragraph{Step 1:} Of course the right-hand side of \eqref{eq:LTtoFD} should be interpreted as a Fredholm series, and we need to make sure that this series is convergent. Observe that $\cd$ is a contour of finite length whereas $\ld$ is not. We will need the following estimate from Abramowitz--Stegun \cite[(6.1.45)]{HandMathFn}: for fixed $x\in\rr$,
\begin{align}
	\label{eq:GammaAS}
	\lim_{\abs{y}\to\infty} \frac{\abs{\Gamma(x+iy)}}{\sqrt{2\pi}}e^{\frac{\pi\abs y}2}\abs{y}^{1/2-x}=1.
\end{align}
Now $\delta_1<\delta_2\wedge \rb{1-\delta_2}$ so we can bound the absolute value of $\frac\pi{\sin\rb{\pi\rb{v-w}}}\cdot \frac1{w-\wt v}$ by a constant (i.e. uniformly in $v,\wt v,w$). Since $\frac1\Gamma$ is an entire function and $v$ runs over a compact domain we can bound $\frac1{\abs{F_s(v)}}$ . Similarly the conditions on the $a_j$ and $\delta_1,\delta_2$ imply that $\abs{\Gamma(v-a_j)}$ is bounded uniformly in $v\in\cd$. Finally, $\abs{s^w}=s^{\delta_2}$ for all $w\in\ld$ so the integrand in the definition of $\klt\rb{v,\wt v}$ can be bounded, for fixed $v,\wt v$ by
\begin{align}
	\label{eq:BoundIntegrand}
	C_3 \abs{\Im(w)}^{\eta+(h-n)(\delta_2-\tfrac12)}\,e^{-\tfrac\pi4(h-n)\abs{\Im(w)}}
\end{align}
for some $C_3>0$ and $\eta\in\rr$. This can easily be seen to be integrable over $w\in\ld$ (in fact, over any vertical line).

\paragraph{Step 2:} $AB=BA$ trick. We now re-write the kernel defining the Fredholm determinant above by using the identity $\det (I+AB)=\det (I+BA)$ for suitable kernels $A,B$. Note that here and throughout we will often abuse notation slighly and blur the distinction between a kernel and the operator it defines. Using this we note that $\klt=AB$ where the kernels defining the operators $A,B$ are given by
\begin{align*}
	 K_A\colon \cd\times\ld\map \rr,\quad\quad K_A\rb{v,w}&=\frac\pi{\sin\rb{\pi\rb{v-w}}}\,\frac{F_s(w)}{F_s(v)}\, \prod_{j=1}^n \frac{\Gamma\rb{v-a_j}}{\Gamma\rb{w-a_j}}\\
	 K_B\colon \ld\times\cd\map\rr,\quad\quad K_B\rb{w,v}&=\frac1{w-v}.
\end{align*}
By the same bounds as above it is easy to see that these define operators $A\colon \Leb^2\rb{\ld}\map \Leb^2\rb{\cd}$ and $B\colon \Leb^2\rb{\cd}\map \Leb^2\rb{\ld}$. Note that the integrals
\begin{align}
	\notag
	\int_{\cd} K_B\rb{w_1,v}K_A\rb{v,w_2}&,\quad\quad\int_{\ld} K_A\rb{v_1,w}K_B\rb{w,v_2}
\intertext{are finite for all $v_1,v_2\in\cd$ and $w_1,w_2\in\ld$ (we checked one of them above, the other is similar). Therefore we can define $\wt K=BA$ as an operator on $\Leb^2\rb{\ld}$ and moreover}
\notag
\det\rb{I+\klt}_{\Leb^2\rb{\cd}} & = \det\rb{I+\kltw}_{\Leb^2\rb{\ld}}.
\intertext{Thus we can write the right hand side of \eqref{eq:LTtoFD} as $\det\rb{I+\kltw}_{\Leb^2\rb{\ld}}$ where}
\label{eq:DefKltW}
\kltw\rb{w,\wt w} &=\int_{\cd} \frac{d v}{2\pi i}\, \frac1{w-v}\,\frac\pi{\sin\rb{\pi\rb{v-\wt w}}}\, \frac{G\rb{\wt w}}{G\rb{v}}
\intertext{and we have defined}
\label{eq:DefG}
G_s(v) &= F_s(v)\prod_{j=1}^n\frac1{\Gamma\rb{v-a_j}} = F_s(v)\prod_{j=1}^n\frac{v-a_j}{\Gamma\rb{v-a_j+1}}.
\end{align}
Here the identity $\Gamma(s+1)=s\Gamma(s)$ has been used for the second equality\footnote{The only reason for re-writing $F_s$ in this way is to isolate the pole of $\frac1{F_s(v)}$}

\paragraph{Step 3:} The integral in \eqref{eq:DefKltW} can be evaluated using residue calculus: the only singularities of the integrand inside the closed contour $\cd$ are simple poles of the form $\frac1{v-a_j}$. Since $\ld$ is a positive distance away from $\cd$ there are no other poles, and the fact that $\delta_1<\delta_2\wedge \rb{1-\delta_2}$ implies that the fraction involving the sine does not have any singularities\footnote{These poles lie inside of the contour thanks to our assumption that $\abs{a_j}<\delta_1$ for all $j$.} inside $\cd$. We can assume for the moment that  the $a_j$ are all distinct; once the following formula has been established the case where some or all of the $a_j$ are equal will follow from continuity.
By computing the residues at the $n$ simple poles we see that 
\begin{align*}
	\kltw\rb{w_1,w_2} & = \frac1{2\pi i}\sum_{j=1}^n f_j\rb{w_1} g_j\rb{w_2}
	\intertext{with $f_j(w)=\frac1{w-a_j}$,}
	g_j(w) & = C_j G(w)\frac\pi{\sin\rb{\pi\rb{a_j-w}}} 
	\intertext{and the constant $C_j\in\rr$ given by}
	C_j &= \frac 1{F_s\rb{a_j}}\,\prod_{\ell\ne j} \Gamma\rb{a_j-a_\ell}.
\end{align*}

\paragraph{Step 4:} Once more AB/BA. This last expression for $\kltw$ can be written as $\kltw=CD$ where $K_C\colon\ld\times\set{1,\ldots,n}\map\rr$ and $D\colon\set{1,\ldots,n}\times\ld\map\rr$ are given by $K_C(w,j)=f_j(w)$ and $D(j,w)=g_j(w)$. We apply once more the AB/BA trick to see that
\begin{align}
	\label{eq:ABBADiscrete}
	\det\rb{I+\klt}_{\Leb^2\rb{\cd}}&= \det\ab{I_n+\int_{\ld}\frac{dw}{2\pi i}\, f_j(w)g_\ell(w)}_{j,\ell=1}^n
\end{align}
where $I_n$ is the $n\times n$ identity matrix (the right hand side corresponds to the Fredholm determinant of the operator $DC$ on $\ell^2\rb{1,\ldots,n}$).

 \paragraph{Step 5:} We now shift the integration contour on the right-hand side of \eqref{eq:ABBADiscrete} from $\ld$ to $-\ell_{\delta_1}$. On the way we will encounter some poles whose residues we will need to evaluate. There is sufficient decay at infinity to justify moving the contours thanks to \eqref{eq:GammaAS}.

Observe that the singularities of $F_s$ are at $-\rb{b_j+M}$ for $M\in\zz_{\geq 0}$. Therefore the condition $b_j>\delta_2$ ensures that we will not cross any of these poles. On the other hand, the sine term in the numerator of $f_j(w)g_\ell(w)$ leads to singularities of the form $w=a_\ell+M$ where $M\in\zz$. We will only cross the pole where $M=0$, i.e. when $w=a_\ell$. Recall that the function $\frac1{\Gamma(\cdot)}$ is entire and has zeroes at the negative integers. Thus, when $j\ne \ell$ the zero at $w=a_\ell$ of $\frac1{\Gamma\rb{w-a_\ell}}$ cancels the simple pole from the sine term. Hence the singularity of the integrand is removable when $j\ne \ell$. 

On the other hand, when $j=\ell$ there is no $\Gamma$ term to cancel the singularity and we obtain a simple pole at $w=a_\ell$, and we now proceed to compute the corresponding residue:
\begin{align}
	\notag
	\Res_{w=a_j}f_j(w)g_\ell(w) & = - F_s\rb{a_j} C_j\frac1{\Gamma(1)}\prod_{r\ne j}\Gamma\rb{a_j-a_r}=-1.
\intertext{Hence}
\label{eq:ShiftContour}
\int_{\ld}\frac{dw}{2\pi i}\, f_j(w)g_\ell(w)&=-\delta_{j\ell}+\int_{-\ell_{\delta_1}}\frac{dw}{2\pi i}\, f_j(w)g_\ell(w)
\intertext{and hence, substituting \eqref{eq:ShiftContour} into \eqref{eq:ABBADiscrete},}
\det\rb{I+\klt}_{\Leb^2\rb{\cd}}&= \det\ab{\int_{-\ell_{\delta_1}}\frac{dw}{2\pi i}\, f_j(w)g_\ell(w)}_{j,\ell=1}^n\\
\label{eq:finalpretransform}
& =\frac1{n!\rb{2\pi i}^n}\int_{\rb{-\ell_{\delta_1}}^n}d \vec w \det\ab{f_j\rb{w_\ell}}_{j,\ell=1}^n\det\ab{g_j\rb{w_\ell}}_{j,\ell=1}^n
\end{align}
where the last equality follows from the Andr\'eiev identity \cite{Andreiev}.

\paragraph{Step 6:} It remains to show that the integrand in \eqref{eq:finalpretransform} is identical to that in \eqref{eq:LT}. For this we will use the \emph{Cauchy determinant identity} \cite[p. 98]{Aufgaben}:
\begin{lem}
	\label{lem:Cauchy}
	Suppose that $x_1,\ldots,x_n, y_1,\ldots, y_n\in\rr$ are all distinct then
	\begin{align}
		\det\rb{\frac1{x_j-y_\ell}}_{j,\ell=1}^n & = \frac{\prod_{j<\ell}\rb{x_j-x_\ell}\rb{y_\ell-y_j}}{\prod_{j,\ell}\rb{x_j-y_\ell}}.
	\end{align}
\end{lem}

\par\noindent Because $\sin(x)=\frac{e^{ix}-e^{-ix}}{2i}$ both determinants in \eqref{eq:finalpretransform} are in the right form to be evaluated using Lemma \ref{lem:Cauchy}. After some rearranging and using the identity $\Gamma(s)\Gamma(1-s)=\frac\pi{\sin\rb{\pi s}}$ we obtain
\begin{align*}
	\det\ab{f_j\rb{w_\ell}}_{j,\ell=1}^n\det\ab{g_j\rb{w_\ell}}_{j,\ell=1}^n &= D_{a,a}D_{w,a}D_{a,a}
\end{align*}
where the terms $D_{a,a},\, D_{w,a}$ and $D_{a,a}$ are respectively given by
\begin{align*}
	D_{a,a} & = e^{in\sum_{j} a_j}\rb{-2\pi i}^{\binom{n}{2}}\, \prod_{j=1}^n \frac1{F_s\rb{a_j}},\\
	D_{a,w} & = e^{-i\pi n \sum_{j}\rb{a_j+w_j}} \rb{2\pi i}^{-n^2} \prod_{j,\ell=1}^n \Gamma\rb{a_j-w_\ell}\\
	D_{w,w} & = e^{i\pi n\sum_j w_j}\rb{2\pi i}^{\binom n2}\,\prod_{a<b}\frac{\sin\rb{\pi\rb{w_a-w_b}}\rb{w_a-w_b}}\pi \prod_{j,\ell} \Gamma\rb{a_j-w_\ell} \prod_{r=1}^n \frac{F_s\rb{w_r}}{F_s\rb{a_r}}
\end{align*}
Performing the apparent cancellations and putting things together leads to the integrand in \eqref{eq:LT}, which completes our proof.


\section{Asymptotics}
\label{sec:asymptotics}

In the previous section (Proposition \ref{prop:LTtoFD}) we saw that 
\begin{align*}
		\int_{(\R_{>0})^n} e^{-s/x_n}\, \mu_{n}(dx) =\det\rb{I+\klt}_{L^2\rb{\cd}}
\end{align*}
where
\begin{align}
	\label{eq:KerLT}
	\klt\rb{v_1,v_2} &= \int_{\ld} \frac{dw}{2\pi i} \frac\pi{\sin\rb{\pi\rb{v_1-w}}}\frac{F_s(w)}{F_s\rb{v_1}}\frac 1{w-v_2}\prod_{j=1}^n \frac{\Gamma\rb{v_1-a_j}}{\Gamma\rb{w-a_j}}
	\intertext{and the function $F_s$ was defined in \eqref{eq:DefF}. From now on we choose $a_j=0$ and $b_j=\gamma$ for all $j$, where $\gamma>0$. Then $1/x_n$ has the same law under $\mu_n$ as the partition function $Z_{m,n}$ of the random
	polymer defined in 	the introduction, taking $m=h-n+1$. We will set $h=\lceil cn\rceil$ for some fixed $c>1$. The correct choice of the parameter $s$, according to the law of large numbers is $s=e^{n\mu-rn^{-1/3}}$ with $\mu$ defined in \eqref{eq:defmuc} below. Then }
	e^{-sZ_{m,n}}&=f_{n,r}\rb{\frac{\ln Z_{m,n}-n\mu}{n^{1/3}}}
\end{align}
	where $f_{n,r}(x)=\exp\set{-e^{n^{1/3}\rb{x-r}}}$. In this section we show that the expectation of the left-hand side above converges, as $n\to\infty$, to a rescaled version of the Tracy--Widom GUE distribution function. Observe that with our choice of parameter $s$ this expectation equals $\det\rb{I+\knr}$ where
\begin{align}
	\label{eq:KerPrelim}
	\knr\rb{v_1, v_2} &=\frac1{2\pi i} \int_{\ell_{\delta_2}} \frac{d w}{w-v_2}\, \frac\pi{\sin\rb{\pi\rb{v_1-w}}}\, \exp\set{n\rb{\hn\rb{v_1}-\hn\rb{w} } - rn^{1/3}\rb{w-v_1}}
	\intertext{and, recalling that $h=\lceil cn\rceil$}
	\hn(z) & = \ln\Gamma\rb{z}-\cadj \ln\Gamma\rb{\gamma + z}+\mu z \ln\Gamma\rb{z+\gamma}+\mu z.
\end{align}
and $\cadj=\frac{\lceil cn\rceil}n$.

	\begin{thm}
		\label{thm:Limit}
		 For $\gamma$ sufficiently small we have 
		 \begin{align*}
		 	\lim_{n\to\infty}\det\rb{I+\knr}_{\Leb^2\rb{\cd}}= F_\text{GUE}\rb{\rb{\ol g/2}^3\, r}
		 \end{align*} 
		 where $\ol g$ was defined in Theorem \ref{main}.
	\end{thm}
	
\par\noindent The proof of Theorem \ref{main} is completed by noting that $f_{n,r}(x)=f_{n,0}(x-r)$ for all $r$ and that $\rb{f_n:=f_{n,0}\colon n\in\nn}$ and $p:= F_\text{GUE}$ satisfy the conditions of Lemma \ref{lem:technical}, whose proof is elementary and can be found in \cite[Lemma 4.1.39]{bc}.

\begin{lem}
	\label{lem:technical}
	For each $n\in\nn$ let $f_n\colon \rr\map [0,1]$ be $f_n$ strictly increasing and converge to $0$ at $\infty$ and $1$ at $-\infty$. Suppose further that for each $\delta>0$, $\rb{f_n\colon n\in\nn}$ converges uniformly to $\textbf{1}_{(-\infty,0]}$. Let $\rb{X_n\colon n\in\nn}$ be real-valued random variables such that for each $r\in\rr$,
	\begin{align*}
		\lim_{n\to\infty}\ee\rb{f_n\rb{X_n-r}} &= p(r)
	\end{align*}
	where $p$ is a continuous probability distribution function. Then $\rb{X_n\colon n\in\nn}$ converges in distribution to a random variable with distribution function $p$.
\end{lem}

\par\noindent It therefore remains to prove Theorem \ref{thm:Limit}. Recall that we need to compute the $n\to\infty$ limit of $\det\rb{I+\knr}_{L^2\rb{\cd}}$ with $\knr$ as defined  in \eqref{eq:KerPrelim} above.


The first step is to identify suitable steepest descent contours to which we will deform the contours $\cd$ and $\ld$. We also introduce the function $\h(z)=\ln\Gamma\rb{z}-c \ln\Gamma\rb{z+\gamma}+\mu z$. Observe that for $z\in\cc$,
\begin{align}
	\label{eq:DifHn}
	\h(z)-\hn(z) & = \rb{\cadj-c} \ln\rb{\Gamma\rb{z+\gamma}}.
\end{align}
and that $\cadj-c=O\rb{n^{-1}}$. For later use we record the first few derivatives of $\h$:
\begin{align*}
	\h'(z) &= \psi(z)-c\psi(\gamma+z)+\mu\\
	\h''(z)&= \psi_1(z)-c\psi_1(\gamma+z)\\
	\h'''(z)&=\psi_2(z)-c\psi_2(\gamma+z)
\end{align*}
where $\psi_k(x)=\frac{d^{k+1}}{dx^{k+1}}\ln\rb{\Gamma(x)}$ is the $k^\text{th}$ polygamma function; in particular $\psi=\psi_0$ is the digamma function. Let $\lambda_c>0$ be small, with the precise value to be chosen later. The proof of the following calculus lemma can be found in Section \ref{sec:proofs}. 

\begin{lem}
	\label{lem:CalcGamma}
	For each $c>0$ and $\gamma>0$ small enough there exists unique $\zcrit$ such that $\h''\rb{\zcrit}=0$. Moreover $\h'''\rb{\zcrit}<0$ and we can write $\zcrit=\gamma\zcritw +O\rb\gamma$ with $\lim_{\gamma\to 0}\zcritw = \frac1{\sqrt{c}-1}$.
\end{lem}

\par\noindent Our asymptotic analysis will consist of shifting our contours to curves that pass through or near $\zcrit$ and showing that in the $n\to\infty$ limit only the parts of the contour near $\zcrit$ survive. We will see that the right choice for $\mu=\mu_c$ is such that $\h'\rb{\zcrit}=0$, i.e. 
\begin{align}
	\label{eq:defmuc}	\mu_c & = c\psi\rb{\gamma+\zcrit}-\psi(\zcrit) \\
	&= \inf_{z>0} \set{c\psi\rb{z+\gamma}-\psi\rb{z}}.
\end{align}
with infimum rather than supremum because $\gc:=-\h'''\rb{\zcrit}>0$. Taylor's theorem implies therefore that, for $v,w$ near $\zcrit$,
\begin{align}
	\label{eq:TaylorH}
	\h\rb{v_1}-\h(w) & = \frac{\gc (w-\zcrit)^3}6 - \frac{\gc (v_1-\zcrit)^3}6 + O\rb{\rb{w-\zcrit}^4}+ O\rb{\rb{v_1-\zcrit}^4}.
\end{align}	
The fact that the lowest power is a cube suggests a scaling of order $n^{1/3}$ around the critical point and we set $\wt v_j=n^{1/3}\rb{v_j-\zcrit}$ and $\wt w=n^{1/3}\rb{w-\zcrit}$. We will see below that only a small part of the integral around the critical point contributes to the limit which leads to

\begin{prop}
	\label{prop:fdlim}
	We have 
\begin{align}
	\label{eq:fdlim}
	\lim_{n\to\infty} \det\rb{I+\knr}_{\Leb^2\rb{\fdcont}}&=\det\rb{1+\klim}_{\Leb^2\rb{\fdclim}}
	\intertext{where}
	\label{eq:fdlimker}
\klim\rb{\wt v_1,\wt v_2}&= \frac1{2\pi i}\int_{\kerclim} \frac{d \wt w}{\wt w-\wt {v_2}}\,\frac1{\vw_1-\ww} \exp\set{\frac{\gc\rb{\wt{w}^3 - \wt{v_2}^3}}6 + r\rb{\wt v_1-\wt w} }
\end{align}
and further $\fdclim=e^{2i\pi/3}\rr_{\geq 0}\cup e^{4i\pi/3}\rr_{\geq 0}$ and $\kerclim=\gamma+\rb{e^{i\pi/3}\rr_{\geq 0}\cup e^{-i\pi/3}\rr_{\geq 0}}$,
see Figure \ref{fig:limconts}.
\end{prop}	

\begin{figure}[H]
 \begin{center}
\begin{tikzpicture}[scale=.3]

\draw[thick] (0,0)--(-5.2,9);
\draw[thick] (0,0)--(-5.2,-9);
\draw[thick] (3,0)--(8.2,9);
\draw[thick] (3,0)--(8.2,-9);

\draw[thin,<-] (0,10)--(0,-10);
\draw[thin,->] (-10,0)--(14,0);

  \node at (-3.6,8) {\small{$\fdclim$}};
   \node at (4.6,8) {\small{$\kerclim$}};
  \node [anchor=north east] at (3.3,0) {\small{$\gamma$}};
\end{tikzpicture}

\footnotesize{\caption{The contours $\fdclim$ and $\kerclim$. The angle between $\fdclim$ and $\kerclim$ and the negative and positive x-axes respectively is given by $\frac\pi3$.}}
	\label{fig:limconts}
\end{center}
\end{figure}

\par\noindent Setting now $v=\rb{\frac{\gc}2}^{1/3}\vw$ and similarly $w=\rb{\frac{\gc}2}^{1/3}\ww$ we obtain $\det\rb{I+\wt{\klim}}_{\Leb^2\rb{\fdcont}}$ where
\begin{align*}
	\wt{\klim}\rb{v_1,v_2} &= \frac1{2\pi i} \int_{\kerclim} \frac{dw}{w-v_2}\,\frac1{v-2}\exp\set{\frac{-\frac{v^3}3+\rb{\frac{\gc}2}^{-1/3}rv}{-\frac{w^3}3+\rb{\frac{\gc}2}^{-1/3}rw}}
\end{align*}
But this is exactly one of the definitions of the Tracy-Widom GUE distribution, see for example Lemma 8.6 in \cite{bcf}.


We begin by deforming the contours $\cd$ and $\ld$ to suitable steepest descent contours. In fact, for $\gamma$ small enough we will be able to do this without passing through any pole. 

The integrand has the following poles in the integration variable $w$:
\begin{itemize*}
	\item $w=v_2$
	\item $w=-M-\gamma$ for $M\in\zz_{\geq 0}$ (these are the poles of $F$)
	\item $w=v_1+2p\pi$ for all $p\in\zz$
\end{itemize*}
On the other hand the poles of the kernel in $v_1,v_2$ are given by
\begin{itemize*}
	\item $v_1=w+2p\pi$ for all $p\in\zz$
	\item $v_2=w$
	\item $v_2=0$
\end{itemize*}

We would like to move the contours $\cd$ and $\ld$ to the following contours, which are illustrated in Figure \ref{fig:SDContours}.

Denote by $C^{w,\pm}$ the line segments of length $\ell-n^{1/3}$ starting at $\zcrit+\gamma n^{-1/3}$ making angles $\frac\pi3$ and $-\frac\pi3$ respectively with the positive $x$-axis and let $\kercont= \rb{\zcrit+\gamma \ell e^{-\pi/3}+i\rr}\cup C^{w,+}\cup C^{w,-}\cup \rb{\zcrit + \gamma \ell e^{\pi/3}+i\rr}$, oriented to have increasing imaginary part.

The closed contour $\fdcont$ is defined differently according to whether $c$ is larger than $\frac52$ or not. For $c>\frac52$ let $\fdcont$ the union of the line segments of length $\frac{6\gamma}{5(\sqrt{c}-1)}$ making angles $\pm\frac{2\pi}3$ with the positive x-axis and the circular segment, centred at $\zcrit$, that connects the end-points of these two segments. For $c\leq \frac{5}2$ we define $\fdcont$ to be the union of the following four line segments: those starting at $\zcrit$ of length $\frac{2\gamma}{c-1}$ making angles $\pm\frac{2\pi}3$ with the positive x-axis and those connecting the end-points of the former with the point $-\frac{2\gamma}{c-1}$. In both cases we give $\fdcont$ the positive orientation.\\

\begin{figure}[H]
 \begin{center}
\begin{tikzpicture}[scale=.3]

\draw[thick] (-12,0)--(-17.2,9);
\draw[thick] (-12,0)--(-17.2,-9);
\draw[dashed] (-23,0)--(-17.2,9);
\draw[dashed] (-23,0)--(-17.2,-9);
\draw[thin,<-] (-20,10)--(-20,-10);
\draw[thin,->] (-25,0)--(-5,0);

 \node at (-14,7) {\small{$\fdcont$}};
 \node at (-5,7) {\small{$\kercont$}}; 
 \node [anchor=north east] at (-10.1,0) {\tiny{$\zcrit$}};
 \node [anchor=south west] at (-9.65,0) {\tiny{$\zcrit+\gamma n^{-1/3}$}};

 \node at (16,7) {\small{$\fdcont$}};
 \node at (25,7) {\small{$\kercont$}}; 
 \node [anchor=north east] at (19.9,0) {\tiny{$\zcrit$}};
 \node [anchor=south west] at (20.35,0) {\tiny{$\zcrit+\gamma n^{-1/3}$}}; 

\node [anchor=north east] at (-22.5,0) {\tiny{$-\frac{2\gamma}{c-1}$}};

\draw[thick] (-10,0)--(-8,3.5);
\draw[thick] (-10,0)--(-8,-3.5);
\draw[dashed] (-8,3.5)--(-8,10);
\draw[dashed] (-8,-3.5)--(-8,-10);

\draw[thin,<-] (12,10)--(12,-10);
\draw[thin,->] (5,0)--(25,0);

\draw[thick] (18,0)--(14,7);
\draw[thick] (18,0)--(14,-7);


\draw [dashed] (14,7) arc (120:240:8.06);

\draw[thick] (20,0)--(22,3.5);
\draw[thick] (20,0)--(22,-3.5);
\draw[dashed] (22,3.5)--(22,10);
\draw[dashed] (22,-3.5)--(22,-10);

\end{tikzpicture}
	\footnotesize{\caption{Contours $\fdcont$ and $\kercont$ for large $c\leq\frac52$ (on the left) and $c>\frac52$ (on the right). The parts $\fdcirrel$ and $\kercontirrel$ of the contours are drawn as dashed lines.}}
	\label{fig:SDContours}
 \end{center}
\end{figure}
\par\noindent It is easy to see that we do not cross any poles of the integrand, further the estimate \eqref{eq:GammaAS} gives sufficient decay at infinity to justify moving the infinite $w$-contour. It follows that $\ee e^{-s/Z_n}=\det\rb{I+\klt}_{\fdcont}$ where
\begin{align}
	\label{eq:KLTMoved}
	\klt\rb{w,\wt w} &= \frac1{2\pi i} \int_{\kercont} \frac{dw}{w-\wt v} \frac\pi{\sin\rb{\pi\rb{v-\wt v}}}\frac{F(w)}{F(v)}\,\prod_{j=1}^n \frac{\Gamma\rb{v-a_j}}{\Gamma\rb{w-a_j}}.
\end{align}


\par\noindent The proof in the rigorous steepest descent analysis now goes along similar lines as, for example, \cite{bcf,bcr}. Fix $\epsilon>0$. We will show that the difference between our formula for the Laplace transform of $\ee e^{-s/Z_N}$ and the right hand side of \eqref{eq:fdlimker} can be bounded by $\epsilon$ for large enough $n$.

\begin{lem}
	\label{lem:FinalStretch} There exists $M^*>0$ such that for $M>M^*$,
	\begin{align*}
		\abs{\det\rb{I+\klimtrunc}_{\Leb^2\rb{\fdclimtrunc}}-\det\rb{I+\klim}_{\Leb^2\rb{\fdclim}}}<\frac\epsilon3
	\end{align*}
	where $\fdclimtrunc=\set{z\in\fdclim\colon \abs{z}\leq M}$,
	\begin{align*}
		\klimtrunc \rb{\wt v_1,\wt v_2}&= \frac1{2\pi i}\int_{\kerclimtrunc} \frac{d \wt w}{\wt w-\wt {v_2}} \exp\set{\frac{\ol g\rb{\wt{w}^3 - \wt{v_2}^3}}6 + r\rb{\wt v_1-\wt w} }
	\end{align*}
	and similarly $\kerclimtrunc=\set{z\in\kerclim\colon \abs{z}\leq M}$.
\end{lem}

\par\noindent From now on we assume that $M>M^*$. Denote by $\fdcrel$ the part of $\fdcont$ consisting of the two line segments starting at $\zcrit$. Similarly let $\kercontrel$ be the corresponding part of $\kercont$. Further define $\fdcirrel=\fdcont\setminus\fdcrel$ and $\kercontirrel=\kercont\setminus \kercontrel$ (see Figure \ref{fig:SDContours}).

\begin{lem}
	\label{lem:BoundH}
	There exist $\gamcrit>0$ and $\ell>0$ such that for $\gamma<\gamcrit$ and $n$ sufficiently large the following hold
	\begin{enumerate}[(i)]
		\item There exists $C_1>0$ such that for $v\in \fdcirrel$,
	\begin{align}
		\label{eq:BoundHCircular}
		\Re\rb{\hn(v)-\hn\rb{\zcrit}}&\leq -C_1.
	\end{align}
		\item There is $C_2>0$ such that for all $v\in\fdcrel$ with $\abs v\geq \ell$,
		\begin{align}
			\label{eq:BoundHvStraight1}
			\Re \ab{\hn(v)-\hn\rb{\zcrit}} & \leq -C_2
		\end{align}
		\item\label{it:Taylor1} There is $C_3>0$ such that for all $v\in\fdcrel$ with $\abs{v}\leq \ell$,
		\begin{align}
			\label{eq:BoundHvStraight2}
			\Re \ab{\hn(v)-\hn\rb{\zcrit}} & \leq -C_3 \Re\ab{\rb{v-\zcrit}^3}
		\end{align}
		\item\label{it:Taylor2} There is $C_4>0$ such that for all $w\in \kercontrel$
	\begin{align}
		\label{eq:BoundHRelW}
		\Re\ab{\hn\rb{\zcrit}-\hn(w)} & \leq -C_4 \Re\ab{\rb{\zcrit-w}^3}.
	\end{align}
 	 \item\label{it:BoundHAwayW} There exists  $C_5>0$ such that for all $\gamma<\gcrit3$ and $w\in \kercontirrel$ 
	\begin{align}
		\label{eq:BoundHAwayW1}
		\Re\ab{\hn\rb{\zcrit}-\hn(w)} & \leq -C_5.
		\intertext{Further there exists $L=L_{c,\gamma}>0$ such that if additionally $\abs w>L$ then}
		\label{eq:BoundHAwayW2}
		\Re\ab{\hn\rb{\zcrit}-\hn(w)} & \leq \frac{(1-c)\pi}4\, \abs{\Im (w)}   
	\end{align}
	\end{enumerate}
\end{lem}

\par\noindent The proof of Lemmas \ref{lem:FinalStretch} and \ref{lem:BoundH} can be found in Section \ref{sec:proofs}. From now on we assume that $\gamma<\gamcrit$. As a first consequence we see the series defining $\det\rb{I+\klt}_{\fdcont}$ is uniformly convergent in $n$, i.e. we may interchange the $n\to\infty$ limit with the series in $k$. 

Thanks to (\ref{it:BoundHAwayW}) the contribution of the $w$ integral along $\kercontirrel$ becomes negligible as $n$ tends to infinity. That is, uniformly in $v_1,v_2\in \fdcont$, as $n\longrightarrow\infty$,
	\begin{align}
		\label{eq:CutIrrelW}
		\int_{\kercontirrel} \frac{d w}{w-v_1}\,\frac\pi{\sin\rb{\pi\rb{v_2-w}}} e^{n\rb{\hn\rb{\zcrit}-\hn(w)}}\ \longrightarrow 0
	\end{align}

\par\noindent Similarly it follows from \eqref{eq:BoundHCircular} and uniform convergence that only the `relevant' part of the $v$-contour survives in the limit. That is, there exists $N\in\nn$ such that for all $n\geq N$,
\begin{align}
	\label{eq:CutIrrelV}
	\abs{\det\rb{I+\klt}_{\Leb^2(\fdcont)} - \det\rb{I+\klt}_{\Leb^2\rb{\fdcrel}}} <\frac\epsilon3.
\end{align}
The estimates form \eqref{eq:BoundHvStraight2} and \eqref{eq:BoundHRelW} now allow us to further discard the parts of $\fdcrel$ and $\kercontrel$ which are further than $Mn^{-1/3}$ away from $\zcrit$ and $\zcrit+n^{-1/3}$ respectively. Now we make the change of variables $v_j=n^{1/3}\vw_j+\zcrit$ and $w_j=n^{1/3} \ww+\zcrit$, and write $\klt\rb{v_1,v_2}=\kltw\rb{\vw_1,\vw_2}$ for $\vw_1,\vw_2\in \fdclimtrunc$. Then
\begin{align*}
	\det\rb{I+\klt}_{\Leb^2\rb{\fdcrel}}=\det\rb{I+\kltw}_{\Leb^2\rb{\fdclimtrunc}}.
\end{align*}
We will show that $\kltw$ converges pointwise to $\klimtrunc$. Once this has been established we can conclude by the DCT and uniform convergence that $\det\rb{I+\klt}_{\Leb^2\rb{\fdcont}}$ converges to $\det\rb{I+\klimtrunc}_{\Leb^2\rangle_M}$. But by Lemma \ref{lem:FinalStretch} this differs only by $\frac\epsilon{10}$ from $\det\rb{1+\klim}_{\Leb^2\rb{\fdclim}}.$ So we have shown that for $N$ sufficiently large, $\gamma$ sufficiently small and $M>M^*$,
\begin{align*}
 \abs{\det\rb{I+\knr}_{\Leb^2\rb{\fdcont}}-\det\rb{1+\klim}_{\Leb^2\rb{\fdclim}}}	& <\epsilon
\end{align*}
subject to establishing pointwise convergence of $\kltw$ to $\klimtrunc$. For this observe that
\begin{align*}
	\frac{d w}{w-v_2} =\frac{d\ww}{\ww-\vw_2},\  n^{-1/3}\frac\pi{\sin\rb{\pi\rb{v_1-w}}} & = \frac1{\vw_1-\ww}+O\rb{n^{-1/3}}, \  rn^{1/3}\rb{w-v_1}=r \rb{\ww-\vw_1}
	\intertext{and, thanks to \eqref{eq:TaylorH} and the fact that $\cadj=c+O\rb{\frac1n}$,}
	n\rb{\hn\rb{v_1}-\hn(w)} &= \frac{\gc}{6} \ab{\rb{\ww-\zcrit}^3-\rb{\vw_1-\zcrit}^3} + O\rb{n^{-1/3}}.
\end{align*}
This concludes the proof of Proposition \ref{prop:fdlim}

\section{Proof of Lemmas}
\label{sec:proofs}

This section is devoted to proving the auxiliary results from Section \ref{sec:asymptotics} above.

\subsection{Proof of Lemma \ref{lem:FinalStretch}}

The last lemma to prove replaces the finite contour $\rangle_M$ by $\fdclim$. By the Dominated Convergence Theorem and continuity of the determinant we have, for sufficiently large $M$,
\begin{align}
	\label{eq:BoundFinal1}
	\abs{\det\rb{I+\klimtrunc}_{\Leb^2\rb{\fdclimtrunc}}-\det\rb{\klim}_{\Leb^2\rb{\fdclimtrunc}}}<\frac\epsilon{6}.
\end{align}
The following useful result can be found as Lemma 8.4 in \bcf.
\begin{lem}
	Let $\Gamma$ be an infinite complex curve and $K$ an integral operator on $\Gamma$. Suppose that there exists $C_1,C_2,C_3>0$ such that $\abs{K\rb{v_1,v_2}}\leq C_1$ for all $v_1,v_2\in\Gamma$ and that
	\begin{align}
		\abs{K\rb{\Gamma\rb{s_1},\Gamma\rb{s_2}}}&\leq C_2e^{-C_3\abs {s_1}}
		\intertext{for all $s\in\rr$ (here, $\Gamma(s)$ denotes the parametrisation of $\Gamma$ by arc length). Then the Fredholm series defining $\det\rb{I+K}_{\Leb^2\rb{\Gamma}}$ is well defined, and for any $\epsilon>0$ there exists $M_\epsilon>0$ such that for all $M>M_\epsilon$,}
		\abs{\det\rb{I+K}_{\Leb^2\rb\Gamma}-\det\rb{I+K}_{\Leb^2\rb{\Gamma_M}}}& \leq \epsilon
	\end{align}
	where $\Gamma_M=\set{\Gamma(s)\colon \abs{s}\leq M}$.
\end{lem}

\par\noindent The proof of Lemma \ref{lem:FinalStretch} is therefore complete if we can find $C_1,C_2>0$ such that $\abs{\klim\rb{v_1,v_2}}\leq C_1e^{-C_2v_1}$ for all $v_1,v_2\in\fdclim$. But this follows immediately from \eqref{eq:fdlimker}.

\subsection{Proof of Lemma \ref{lem:CalcGamma}}

Convexity considerations show that if there exists a zero of $\h''$ then it is unique. Let us write $z=\gamma \wt z$ then
\begin{align*}
	\h''(z)&=\gamma^{-2}\rb{\frac1{\wt z^2}-\frac c{\rb{1+\wt z}^2}} + \frac{\pi^2}6 \, (1-c)+ O(\gamma)
\end{align*}
with the error being uniform in $\wt z$ over compact intervals. Hence, for $\gamma$ small enough  we have $\h''\rb{\frac\gamma{\sqrt c-1}}<0$ and $\h''\rb{\frac\gamma{\sqrt c-1}-\lambda_c}>0$, from which the result follows.

\subsection{Proof of Lemma \ref{lem:BoundH}}

The following small $\gamma$ estimates will be useful. Throughout we set $z=\gamma \zw$.

\begin{lem}
	\label{lem:ExpandMu}
	There exist $\rb{\mucritw\colon \gamma>0}$ such that
	\begin{align}
		\label{eq:ExpandMu}
		\mucrit& =\frac{\mucritw}\gamma + O\rb{\gamma^{-2}}
	\end{align}
	and $\mucritw \longrightarrow -\rb{\sqrt c-1}^2$ as $\gamma\longrightarrow 0$.
\end{lem} 
\begin{proof}
	We have
	\begin{align*}
		\mucrit & = c\Psi\rb{\gamma\rb{\zcritw+1}}-\Psi\rb{\gamma \zcritw}\\
		& = \frac1\gamma\ab{\frac1\zcritw - \frac c{1+\zcritw}}+O\rb{\gamma^{-2}}.
	\end{align*}
	The claim now follows from Lemma \ref{lem:CalcGamma}.
\end{proof}

\par\noindent We also record the following small $\gamma$ expansions.
\begin{align}
	\label{eq:ExpandGc}	
	\gc & = -\h'''\rb{\zcrit}= \gamma^{-3}\rb{\frac{2c}{(1+\zcritw)^3}-\frac2{\rb{\zcritw}^3}}\\
		\label{eq:ExpandHDif}
	\h(z)-\h(v) & = c\log\rb{1+\vw} - c\log\rb{1+\zw} - \log\rb{\vw} + \log\rb{\zw} \\
	\notag
	& \quad+ \mucritw\rb{\vw-\zw}+O(\gamma)
\end{align}

\begin{proof}[Proof of Lemma \ref{lem:BoundH}.]
	\emph{(i)} Because $v$ varies over a compact set it follows from \eqref{eq:DifHn} that there exists some $C>0$ such that
	\begin{align*}
		\abs{\hn(v)-\hn\rb{\zcrit}-\ab{\h(v)-\h\rb{\zcrit}}}<\frac Cn
	\end{align*}
	holds for all $v\in\fdcont$. Therefore we may as well prove the claim with $\hn$ replaced by $\h$, which is what we will do.

	Since the contours are different we will consider the cases $c>\frac52$ and $c\leq \frac52$ separately.

\paragraph{Case I: $c>\frac52$.} Fix $\ep>0$ to be chosen later and write $v=\gamma \vw$. Recall that $\zcrit=\gamma\zcritw$ and $\mucrit=\frac\mucritw\gamma$. By \eqref{eq:ExpandHDif},
\begin{align*}
	\h\rb{v}-\h\rb{\zcrit}& = c\ln\rb{\frac{\vw+1}{\zcritw+1}} - \ln\rb{\frac{\vw}{\zcritw}} + \mucritw\rb{\vw-\zcritw} + O(\gamma)
		\intertext{where the error term is uniform in $\vw$ (because the latter varies over a compact contour). Now $\vw=\zcritw+r_{c} e^{i\theta}$ where $r_c=\frac65\,\frac{1}{\sqrt{c}-1}$ and $\theta\in\ab{\frac{2\pi }3,\frac{4\pi}3}$, so we obtain, for $\gamma$ small enough,}
	\Re\rb{\h\rb{v}-\h\rb{\zcrit}}& \leq c\ln\abs{1+ \frac{r_c}{\zcritw+1}\, e^{i\theta}} - \ln\abs{1+\frac{r_c}{\zcritw}\,e^{i\theta}} + \Re\rb{\mucritw r_c e^{i\theta}} + \frac\ep2.
	\intertext{By Lemmas \ref{lem:CalcGamma} and \ref{lem:ExpandMu}  we can now ensure, by choosing $\gamma$ small enough, that}
		\Re\rb{\h\rb{v}-\h\rb{\zcrit}}& \leq c\ln\abs{1+ \frac6{5\sqrt{c}}\, e^{i\theta}} - \ln\abs{1+\frac65\,e^{i\theta}} -\frac65\rb{\sqrt{c}-1} \cos(\theta) + \ep\\
		&= \frac c2\ln\ab{\rb{1+\frac65\sqrt{c}\, \alpha]}^2 + \frac{36}{25c}\,(1-\alpha^2)}\\
		& \quad - \frac12\ln\ab{\rb{1+\frac65\alpha}^2+\frac{36}{25}(1-\alpha)} - \frac65\rb{\sqrt{c}-1}\alpha + \ep
\end{align*}	
where we have written $\alpha=\cos(\theta)\in\ab{-1,-\frac12}$. Denote by $f(\alpha,c)$ the last expression above, with $\ep=0$. For each $c>\frac52$ the function $\alpha\longmapsto f(c,\alpha)$ has a unique critical point on the interval $\ab{-1,-\frac12}$ which turns out to be a minimum. Thus we are reduced to consider the end-points. Now $f\rb{\cdot,-\frac12}$ is strictly decreasing and clearly $C_{11}=f\rb{\frac52,-\frac12}<0$. On the other hand $f\rb{\cdot,-1}$ is strictly increasing and tends to $C_{12}=\ln(5)-\frac{48}{25}<0$ as $c\longrightarrow \infty$. Taking now $C_1=\ep=\frac12\min\set{C_{11},C_{12}}$ completes the proof for the case $c>\frac52$.

\paragraph{Case II: $c<3$.} The contour in question is the union of the (complex) line segments  $\ab{\frac{2\gamma}{c-1}e^{2i\pi/3},-\frac{2\gamma}{c-1}}$ and $\ab{\frac{2\gamma}{c-1}e^{-2\pi i/3},-\frac{2\gamma}{c-1}}$. By symmetry it suffices to consider the former. Thus, writing $v=\gamma\vw$,
\begin{align}
	\label{eq:AwayVcSmall}		
	\vw=t+i\frac{\sqrt 3}{\sqrt{c}+2}\rb{t+\frac{2}{c-1}},\quad\quad t\in \ab{-\frac{2}{c-1},\frac{\sqrt c}{c-1}}.
\end{align}
 Fix $\epsilon>0$. By Lemmas \ref{lem:CalcGamma} and \ref{lem:ExpandMu} as well as \eqref{eq:ExpandHDif} and \eqref{eq:AwayVcSmall} we have, for $\gamma$ small enough and then $n$ large enough,
 \begin{align*}
 	\Re\ab{\h\rb{v}-\h\rb{\zcrit}}&\leq (c-1)\ln\rb{\sqrt c-1} + \ln\rb{\sqrt{c}} -\frac1{(\sqrt{c}-1)^2} \rb{t-\frac1{\sqrt{c}-1}} \\
	& \quad - \ln \ab{t^2+ \frac{3}{\rb{\sqrt c+2}^2}\rb{t+\frac{2}{c-1}}^2}\\
	& \quad + c \ln \ab{(t+1)^2+ \frac{3}{\rb{\sqrt c+2}^2}\rb{t+\frac{2}{c-1}}^2} + \epsilon
 \end{align*}
 Temporarily denote the right hand side above by $F(c,t,\ep)$. For any fixed $c\in\left(1,\frac52\right]$ the function $F(c,\cdot,0)$ has a unique critical point on the interval $\ab{-\frac{2}{c-1},\frac{\sqrt c}{c-1}}$, at which point the second derivative is positive. Furthermore it is easy to verify that the end points $t_0=-\frac{2}{c-1}$ and $t_1=\frac{\sqrt c}{c-1}$ satisfy $F(c,t_0,0)<0$ and $F(c,t_1,0)<0$. This completes the proof for $c\leq \frac52$ and hence for part (i) of the lemma.\\

\par\noindent \emph{(ii)}	By the same argument as in part (i) we may replace $\hn$ by $\h$.

Consider first the case where $c>\frac 52$, so that we have $v=\gamma\rb{\zcritw+\frac{r e^{2i\pi/3}}{\sqrt{c}-1}}$ for $r\in\ab{0,\frac 65}$. Since $\cos\rb{\frac{2\pi}3}=-\frac12$,
	\begin{align*}
		\Re\ab{\h(v)-\h\rb{\zcrit}} &  = c\ln\abs{1+\frac{r e^{2i\pi/3}}{(\sqrt c-1)(\zcritw+1)}} - \ln\abs{1+\frac{re^{2i\pi/3}}{\zcritw(\sqrt c-1)}} - \frac12 \frac{r\mucritw}{\sqrt{c}-1}
		\intertext{Fix $\epsilon>0$. Using \eqref{eq:ExpandHDif} and Lemmas \ref{lem:CalcGamma} and \ref{lem:ExpandMu} as above we have, for $\gamma$ suitably small,} 
		\Re\ab{\h(v)-\h\rb{\zcrit}} & \leq c\ln\ab{\rb{1+ \frac{r}{2\sqrt c}}^2 +  \frac{3r^2}{4\sqrt c}} - \ln\ab{\rb{ \frac r2}^2+\frac{3r^2}4}\\
		& \quad+\frac{r}2\rb{\sqrt{c}-1} + \ep.
	\end{align*}
	Now for any fixed $r$ the right hand side is decreasing in $c$, so it is enough to consider the case where $c=\frac52$, for which it is easy to see that the quantity above is bounded above away from zero (for small enough $\epsilon$), uniformly in $r\in\ab{0,\frac65}$. Taking $\gamma$ small enough deals with the error term (which is uniform in the other variables involved).\\

For the case $c\leq \frac52$ we set $v=\gamma\rb{\zcritw + \frac{2r}{c-1}\,e^{2i\pi/3}}$ with $r\in[0,1]$. A similar computation as in the case $c>\frac52$ shows that for this choice of $v$ (at any fixed $r\in[0,1]$) the function $c\longmapsto \Re\ab{\h(v)-\h\rb{\zcrit}}$ is strictly increasing in $c$ and converges to zero as $c\longrightarrow 1$. Thus the claim holds for any fixed $c>1$, as required.
		
Parts (\ref{it:Taylor1}) and (\ref{it:Taylor2}) follow from Taylor's theorem and the fact that $\hn\rb{\zcrit}=\hn'\rb{\zcrit}=0$.

It remains to prove part (\ref{it:BoundHAwayW}). For the first assertion observe first that by \eqref{eq:DifHn} we have, uniformly in $w\in\kercontirrel$,
	\begin{align*}
		\Re\left[\hn\rb{\zcrit} - \hn\rb{w}\right. &- \left.\rb{\h\rb{\zcrit}-\h(w)}\right]\\
		  & = \underbrace{\Re\ab{\hn\rb{\zcrit}-\h\rb{\zcrit}}}_{O\rb{n^{-1}}} + \Re\ab{\h\rb{w}-\hn(w)}\\
		&= \rb{\cadj-c}\ln\abs{\Gamma\rb{w+\gamma}} + O\rb{n^{-1}}
	\end{align*}
	Now $\abs{\Gamma\rb{w+\gamma}}<1$ for $w\in\kercontirrel$ and we have chosen $\cadj>c$, so the first summand above is negative and we can once more reduce to the case where $\hn$ is replaced by $\h$. Next, write $w=\gamma\ww$ so that $\ww=\zcritw+\ell e^{i\pi/3}+iy$ for $y\geq 0$ or $\ww=\zcritw+\ell e^{-i\pi/3}+iy$ for $y\leq 0$. By symmetry it is enough to consider the former case. Fix $\ep>0$. Applying once more Lemmas \ref{lem:CalcGamma} and \ref{lem:ExpandMu} and \eqref{eq:ExpandHDif} as well as the fact that $e^{i\pi/3}=\frac12+i\, \frac{\sqrt 3}{2}$ we get, for suitably small $\gamma$,
	\begin{align*}
		\Re\ab{\h\rb{\zcrit}-\h(w)} & \leq \frac12\ln\ab{\rb{1+\frac{\ell\rb{\sqrt{c}-1}}{2}}^2 + \rb{\sqrt{c}-1}^2\rb{y+\frac{\sqrt{3}\ell}2}^2} \\
		& \quad - \frac{c}2\ln\ab{\rb{1+\frac{\ell\rb{\sqrt{c}-1}}{2\sqrt{c}}}^2+\frac{\rb{\sqrt{c}-1}^2\rb{y+\frac{\sqrt{3}\ell}{2}}^2}{c}}\\
		& \quad - \frac{\rb{\sqrt{c}-1}^2\ell}2 + \ep
	\end{align*}
	Let us denote by $F\rb{c,\ell,y,\ep}$ the last expression above. It is straightforward to check that the map $y\longmapsto F\rb{c,\ell,y,0}$ is strictly decreasing on $[0,\infty)$. Furthermore the map $\ell\longmapsto F\rb{c,\ell,0,0}$ is strictly decreasing on $[0,\infty)$ and moreover $F(c,0,0,0)=0$. Since $\ell>0$ it follows that there exists $\wt{C}_5>0$ such that $F\rb{c,\ell,y,0}\leq -\wt{C}_5$ for all $c>1$ and $y\geq 0$. The first assertion now follows by choosing $\epsilon=C_5=\frac12\wt{C}_5$.
		
	For the second assertion we will apply the bound \eqref{eq:GammaAS}: for any $\eta>0$,
	\begin{align*}
		\Re\ab{\hn\rb{\zcrit}-\hn(w)}&\leq C- \ln\abs{\Gamma(x+iy)}+\cadj\ln\abs{\Gamma(\gamma+x+iy)}\\
		& \leq C - \ln\rb{\sqrt{2\pi}(1+\eta)e^{-\pi\abs{y}/2} \abs{y}^{x-\tfrac12}} \\
		& \quad+\cadj\ln\rb{\sqrt{2\pi}(1-\eta)e^{-\pi\abs{y}/2}\abs{y}^{x+\gamma-\tfrac12}}\\
		& \leq C+\rb{1-\cadj}\, \frac\pi2 \abs{y}
	\end{align*}
	from which the estimate follows by observing that $\cadj\in\ab{c, c+\frac1n}$.
\end{proof}

\bibliography{biblio}

\end{document}